\documentclass[a4paper,twoside, 10pt]{article}
\usepackage[english]{babel}  
\usepackage{amsmath}
\usepackage{amsthm}
\usepackage{amssymb}
\usepackage[T1]{fontenc}
\usepackage{a4wide}
\usepackage{tikz}
\usepackage[bookmarks]{hyperref}
\usetikzlibrary{patterns, arrows}

\newtheorem{teorema}{Teorem}[section]
\newtheorem{proposicion}[teorema]{Proposition}
\newtheorem{comentario}[teorema]{Remark}
\newtheorem{example}[teorema]{Example}
\newtheorem{defi}[teorema]{Definition}
\newtheorem*{remark}{Remark}
\usepackage[absolute]{textpos}

\title{Schmüdgen's theorem and results of positivity}
\author{Christoph Schulze}
\date{}

\begin{document} 

\begin{textblock*}{15cm}(3.2cm,1.13\textheight)
\scriptsize
\noindent\rule{\textwidth}{0.1pt}
This paper is an english translation of a final paper I made during my Erasmus semester (Febuary to July 2014) at Universidad Complutense de Madrid with some little changes and corrections and a remark at the end. I want to thank Antonio Díaz-Cano Ocaña for the supervision of this work and his helping comments.
\end{textblock*}

\maketitle

\huge \textbf{Abstract}\\

\normalsize
The object of this final paper is the presentation of the theorem of Schmüdgen and other results, which are certificates of positivity on basic semialgebraic sets. Schmüdgen's result from 1991 leads to presentations without denominators for positive polynomials, when the semialgebraic set is compact. So our top priority is to show the theorem of Schmüdgen and present similar results in more general cases, especially when the set is not compact or when we are working in projective spaces.\\
\indent In contrast to presentations in the literature our approach differs in two respects. The first difference is that we emphasize the relationship to the corresponding homogeneous or projective problem, which will lead us to some generalizations of the theorems of Schmüdgen and Putinar. The second difference is, that we want to present our results in a self-contained manner using only a minimum of algebraic structure. Of course, this leads on one hand to proofs, which apply only in special geometric situations, but with our inductive property we get a good geometric presentation, which at the same time generalizes easily to our projective proofs.\\
\indent The paper is organized as follows. In the first part we present, after a short historical introduction, the basic notation together with fundamental relations between the affine and homogeneous and projective problems. This is followed in the second part by some basic elementary results, which will be necessary for our proof of the theorem of Putinar, but which are also interesting Positivstellensätze. The third part begins with the mentioned inductive property, which leads us directly to the theorem of Putinar. We state some theorems to illustrate when the condition of the Putinar's theorem holds and state our first projective result. Up to this point our work is self-contained. The need for stronger tools is caused by the existence of zeros. So we need the Positivstellensatz from Krivine and Stengle to get, together with our prior work, a proof of the theorem of Schmüdgen. To obtain a homogenous version of Schmüdgen's theorem we use a proof originally given by Wörmann together with a generalization of the archimedean property and a homogeneous Positivstellensatz from Zeng. With the inductive property we get at the end of part three a projective theorem of Putinar. Our last chapter deals with the non-compact case. We present the result in dimension 1 and the negative result in dimension 3, a stability theorem of Netzer and some positive examples of Marshall generalized by Nguyen and Powers and new examples from Vuy and Toan.\\
\indent All in all we get a detailed presentation of the theorems of Putinar and Schmüdgen with generalizations to the projective and homogeneous cases and we present the most important results that have been proved so far in the non-compact case.\\


\newpage

\section{Introduction}

\subsection{History}

A problem with certificates of positivity already appears in the problems of Hilbert of 1900. Problem 17 asks if every non-negative polynomial is a quotient of sum of squares of polynomials. This problem has a positive answer, given by Artin in 1927. Tarski developed his work and proved in 1948 his transfer principle which is a very strong tool and which has been used by Krivine in 1964 and Stengle in 1974 to show their Positivstellensatz for basic semialgebraic sets.\\
Up to 1991 the most important Positivstellensätze had been using denominators (or factors which is the same). So the theorem of Schmüdgen, who worked on the moment problem and obtained in his investigations a Positivstellensatz, was a surprise. In this theorem we have a compact set in $\mathbb{R}^n$ described by polynomials of $\mathbb{R}[x_1,\dots,x_n]$, namely, the compact set is the set of points where each polynomial is non-negative. The proposition is that each strict positive polynomial on the set is a combination of sums and products of squares of polynomials and the polynomials which describe our set. Clearly it is a certificate of non-negativity on the set and it does not need denominators. Soon afterwards Putinar asked how many products are necessary in the presentation and this led to further investigations.\\
Another question is, whether one can generalize the theorem to non-compact sets. In this case we have so far only an answer in dimension $1$ and a negative answer in dimension $3$ or higher. In dimension $2$ there are known some positive results and negative results, but no complete answer. As in Schmüdgen's theorem, investigations so far seem strongly connected to the moment problem, but questions concerning effective calculation methods of representations are also important because of connections with optimization.\\
Positivity questions for polynomials have a long history and new results with connetions to other parts of mathematics continue to be found.

\subsection{Affine, homogeneous and projective problems, change of coordinates and moment problem}

In this part we will present the most important notations of this work with the connections between them.\\
We will work in a ring of polynomials over $\mathbb{R}$, we fix $n$ for the number of variables and we will use the abbreviation $\mathbb{R}[X]:=\mathbb{R}[x_1,\dots,x_n]$. We also need sums of squares of polynomials, which we will notate by $\sum \mathbb{R}[X]^2$.\\
We will use the words positive, non-negative, negative and non-positive every time for the relations $>0$, $\geq 0$, $<0$ and $\leq 0$ which shall be fulfilled on the corresponding set and when we do not desribe a certain set we require, that they are fulfilled on the whole space. \\
Let $S=\left\{ g_1, \dots, g_s \right\}\subset\mathbb{R}\left[X\right]$ be a finite set of polynomials (we also fix $s$ for the number of elements of $S$). We name
\begin{itemize}
\item[] $K_S:=\left\{x \in \mathbb{R}^n | g_i(x)\geq 0, i=1,\dots , s\right\}$ semialgebraic set defined by $S$
\item[] $M_S:=\left\{\sum_{i=0}^{s} \sigma_i g_i | \sigma_i \in \sum \mathbb{R}[X]^2, g_0:=1 \right\}$ quadratic module defined by $S$
\item[] and $T_S=\left\{\sum_{\alpha\in \left\{0,1\right\}^s} \sigma_{\alpha} g^{\alpha} | \sigma_{\alpha} \in \sum \mathbb{R}[X]^2\right\}$ preorden defined by $S$. 
\end{itemize}
Clearly polynomials of $M_S$ and $T_S$ are non-negative on $K_S$. Schmüdgen's theorem shows that we have representations with element of $T_S$ and in the theorem of Putinar we achieve representations with elements of $M_S$.\\

We will also need notations in the homogeneous and projective case. We use sometimes $\mathbb{R}[X]=\mathbb{R}[x_1,\dots,x_n]$ for homogeneous problems, when our object is treating infinity in the affine problem, but in general we use the notation $\mathbb{R}[X_0]=\mathbb{R}[x_0,x_1,\dots,x_n]$. We claim that in the sums of squares only appear squares of (homogenous) polynomials of the same degree and we write $\sum \mathbb{R}[X]^2$ or $\sum \mathbb{R}[X_0]^2$.
Let $S=\left\{ g_1, \dots, g_s \right\}\subset\mathbb{R}\left[X_0\right]$ be a finite set of homogeneous polynomials. We name
\begin{itemize}
\item[] $K_S^h:=\left\{x \in \mathbb{R}^{n+1}\backslash \left\{0\right\} | g_i(x)\geq 0, i=1,\dots , s\right\}$ semialgebraic set defined by $S$
\item[] $M_S^h:=\left\{\sum_{i=0}^{s} \sigma_i g_i | \sigma_i \in \sum \mathbb{R}[X_0]^2_h, g_0:=1, deg(\sigma_i g_i)=deg(\sigma_j g_j), i,j\in \left\{0,\dots,s\right\}\right\}$\\ homogeneous quadratic module defined by $S$
\item[] y $T_S^h=\left\{\sum_{\alpha\in \left\{0,1\right\}^s} \sigma_{\alpha} g^{\alpha} | \sigma_{\alpha} \in \sum \mathbb{R}[X_0]^2_h, deg(\sigma_{\alpha} g^{\alpha})=deg(\sigma_{\beta} g^{\beta}), \alpha, \beta\in \left\{0,1\right\}^s\right\}$\\ homogeneous preorden defined by $S$.
\end{itemize}
The homogeneous semialgebraic set is, except the origin, the same like in the affine case and the elements of $M_S^h$ and $T_S^h$ are non-negative on $K_S^h$.\\
The projective case is a special case of the homogeneous case. We claim that polynomials of $S$ (are homogenous and) have even degree. For homogeneous polynomials of even degree it is well defined to speak of positivity (non-negativity, \dots) of the polynomial in the projective space, since the factor of the equivalence class is a square (points, except of the origin, of a straight line through the origin of $\mathbb{R}^{n+1}$ form an equivalence class and these are the points of $\mathbb{R}P^n$). Another sort of imagination of the projective space is the unit sphere with opposite points identified. In this case it makes sense to interpret polynomials as functions and we will use the topology which is induced of it.\\ 
To achieve a formal definition of positivity, we introduce again semialgebraic sets. We name
\begin{itemize}
\item[] $K_S^p:=\left\{x \in \mathbb{R}P^n | g_i(x)\geq 0, i=1,\dots , s\right\}$\\ projective semialgebraic set defined by $S$.
\end{itemize}

Homogenization not behave well with representations with squares when we have polynomials of odd degree. So we use homogenization to even degree to transform affine problems in projective problems (we use the notation $^p$) and dehomogenization to transform projective problems in affine problems. Affine representations correspond to a equivalence class of projective representations. Nevertherless the concept of positivity not behave well in infinity. We will demonstrate it with an example:

\begin{example}
We consider $S=\left\{x_1, 1-x_1^2-x_2^2\right\}$. Because of Putinar's theorem we get for every $f>0$ on the set $K_S$ sums of squares $\sigma_0, \sigma_1, \sigma_2$ such that $f=\sigma_0 +\sigma_1 x_1+ \sigma_2 (1-x_1^2-x_2^2)$.\\
When we are homogenizing $S$ at even degree we have $S^p=\left\{x_0 x_1, x_0^2-x_1^2-x_2^2\right\}$ and Putinar's theorem in affine language tells us that for every homogeneous polynomial $f$ of even degree, which is positive on $K_{S^p}^p$, exist $k, k_0, k_1, k_2\in \mathbb{N}_0$ such that $x_0^{2k}f=x_0^{2k_0}\sigma_0^p+x_0^{2k_1}\sigma_1^p x_0 x_1+x_0^{2k_2}\sigma_2^p (x_0^2-x_1^2-x_2^2)$. We can claim that one exponent of $x_0$ equals $0$ but we can also speak of equivalence classes of representations respective to $x_0$.\\
Considering $S=\left\{x_1, x_2, 1-x_1-x_2\right\}$ we have $S^p=\left\{x_0 x_1, x_0 x_2, x_0^2-x_0x_1-x_0x_2\right\}$ and the set $K_S^p$ contains $\left\{(x_0:x_1:x_2)\in \mathbb{R}P^2|x_0=0\right\}$. Otherwise $T_S$ contains because of to Schmüdgen's theorem $2-x_1^2-x_2^2$ with the homogenization $2x_0^2-x_1^2-x_2^2$, which is negative on the set described by $x_0=0$.\\ In the case when we translate projective Positivstellensätze to affine theorems, we obtain theorems, which not have to hold for all polynomials which are positive on the corresponding affine set $K_S$, when we have a condition on infinity (especially for polynomials of odd degree).
\end{example}

Spaces appear every time with their symmetries. For example Schmüdgen's theorem is invariant to affine transformations. If one translates Schmüdgen's (or Putinar's) theorem to projective language it is not invariant to changes of coordinates in projective space. If one homogenizes to even degree with $x_0$ and dehomogenizes with $x_1$ one introduces denominators of $x_0$ and if $x_1$ does not divide $f$ the degree of the right side of our equation can not be bigger than the degree of $f$ with the potency of $x_0$. In part $3$ we will show a homogeneous version of these theorems.\\

Now we have introduced the most important concepts and we present the moment problem.
The moment problem asks, if for every linear functional $L:\mathbb{R}[X]\rightarrow \mathbb{R}$ with $L(1)=1$ and a closed set $K$ exist a (positive) Borel measure $\mu$ which support is contained in $K$ and which fulfills $L(f)=\int_K f d\mu$ for all $f\in \mathbb{R}[X]$.
A theorem of Haviland tells us, that such a measure exists if and only if $L(f)\geq 0$ for all $f$ with $f\geq 0$ on $K$. Representations of positive polynomials (especially without denominators) are important to reduce the number of polynomials which have to fulfill that condition.\\
In the case of compact basic semialgebraic sets Schmüdgen's theorem shows that such representations exist. Another property for semialgebraic sets is the strong moment property (SMP) which has a semialgebraic set described by $S$, if for every polynomial $f$ non-negative on $K_S$ and every linear functional $L$, with $L(g)\geq 0$ for every $g\in M_S$, we have $L(f)\geq 0$. Presentations without denominators imply SMP such that if a set does not have SMP it can not hold, that every positive polynomial is in the quadratic module. So negative conditions for SMP lead to negative conditions of the question of this paper.

\subsection{Structure of this paper and more notations}

The object of this paper is to give a demonstration of Schmüdgen's theorem and present new results of positivity and non-negativity with some examples.\\
In the following part of this paper we will start with proofs of first theorems of positivity in a elementary manner. Nevertheless, these theorems are important Positivstellensätze with modern proofs and will be the base of our proof of Putinar's theorem in part $3$ of this paper.\\
In that part we will continue to make our proofs most elementary and we will show in that way Putinar's theorem and a projective version. The most important proposition of that part will be the inductive property. It is followed by a proof of Schmüdgen's theorem, for which we will need the Positivstellensatz of Krivine and Stengle, which we will claim in this paper. A generalization of the Positivstellensatz of Krivine and Stengle to a homogeneous Positivstellensatz of Zeng permits us to show a homogeneous version of Schmüdgen's theorem together with a projective version of Putinar's theorem.\\ At the end we will come to the fourth part, which treats the non-compact case. There we will mention some results in dimension $1$ and dimension $3$ or higher dimensions and we will describe the most important known examples in dimension $2$ with some examples.\\

In this work we will use multi-index notation with its operations to make the notation easier. We use $X$ and $X_0$ also for the vectors $(x_1,\dots,x_n)$ and $(x_0,\dots,x_n)$ or the arguments of a polynomial and so $X_0\geq 0$ is the semialgebraic set defined by $x_0\geq 0, \dots, x_n\geq 0$. When we describe semialgebraic sets, we will normally use $x$ in the affine case and $X_0$ in the homogeneous and projective case as variable in the arguments. When we speak of a homogeneous identity we want to say, that all polynomials which appear are homogeneous and if we calculate both sides of the equation we never add polynomials of different degrees like in the representations of $M_S^h$ and $T_S^h$. In part 4 we use $x$ instead of $x_1$ in the case $n=1$ and $x$ and $y$ instead of $x_1$ and $x_2$ in the case $n=2$. Sums which appear will be every time finite sums.

\newpage

\section{Pólya's theorem and applications}

\subsection{Pólya's theorem}

Pólya's theorem of $1928$ is one of the first most important theorems about positive polynomials on a set and has elementary proofs. One can find a proof in the paper "`An effective version of Pólya's theorem on positive definite forms"' by Jesús A. de Loera and Francisco Santos \cite{LS}.

\begin{teorema}{\textbf{(Pólya)}}\\
Let $f\in \mathbb{R}[X_0]$ be a homogeneous polynomial of degree $d$ which is positive on the first octant $\left\{X_0|X_0\geq 0\right\}$ except on the origin. Thus there is a $N_0\in \mathbb{N}$ such that in the homogeneous polynomial $(x_0+\dots+x_n)^N f$ every monomial of degree $d+N$ has positive coefficient for $N\geq N_0$.
\end{teorema}

\subsection{Habicht's theorem}

In 1940 Habicht proved in his paper "`Über die Zerlegung strikte definiter Formen in Quadrate"' \cite{H} with Polya's theorem a theorem which is a special case of the 17th problem of Hilbert. It says that each homogeneous polynomial, which is positive except on the origin, one can write as a quotient of two sums of squares of homogeneous polynomials. Following his work we show a bit stronger version. In the article \cite{LS} one can find a more general version.

\begin{teorema}{\textbf{(Habicht)}}\\
Let $f\in \mathbb{R}[X_0]$ be a homogeneous polynomial of degree $2d$, which is positive except on the origin (=positive on $\mathbb{R}P^n$). Thus exist sum of squares $M_1$, $R_1$, $M_2$ and $R_2$, which are positive except on the origen, such that $M_1$ and $M_2$ are sums of squares of monomials and $(M_2+R_2) f=M_1+R_1$.
\end{teorema}

\begin{proof}
First we notice that because of Pólya's theorem every homogeneous polynomial, which only depends on $x_0^2, \dots, x_n^2$ and which is positive except on the origin one can multiply with a potency of $(x_0^2+\dots+x_n^2)$, such that in the product only appear monomials, which are squares with positive coefficients.\\
We consider the set of signs $\tau=(\tau_0,\dots, \tau_n) \in \left\{-1,+1\right\}^{n+1}$ and the polynomials $f_{\tau}:=f(\tau_0 x_0,\dots, \tau_n x_n)$, which are also positive except on the origin. Let $s_1,\dots, s_{2^{n+1}}$ be the elementary symmetric polynomials of the polynomials $f_{\tau}$ with $\tau \in \left\{-1,+1\right\}^{n+1}$, such that $s_i$ has degree $2di$ ($\forall i\in \left\{1,\dots, 2^{n+1}\right\}$). Because of their definition $s_1,\dots, s_{2^{n+1}}$ are polynomials which do not change, when we change a sign of a variable. So they only depend on $x_0^2, \dots, x_n^2$.\\
When we apply Polya's theorem like mentioned at the same time (and multiplying with $x_0^2+\dots+x_n^2$ some times more for the degree) we obtain sums of squares of monomials $\sigma_1,\dots, \sigma_{2^{n+1}}$, such that $\sigma_i$ corresponds to $s_i$ and such that there is a $D$ with $\sigma_i$ has degree $2Di$ ($\forall i\in \left\{1,\dots, 2^{n+1}\right\}$).\\
Let $\tilde{f}:=(x_0^2+\dots+x_n^2)^{D-d} f$. With Vieta's theorem we have the equation $$\tilde{f}^{2^{n+1}}+\sum_{i=1}^{2^{n+1}} (-1)^{i}\sigma_i \tilde{f}^{2^{n+1}-i}=0.$$
If we write summands of same sign on the same side we obtain $\tilde{f}^{2^{n+1}}+\sum_{i=1}^{2^n} \sigma_{2i} \tilde{f}^{2^{n+1}-2i}=\tilde{f} \sum_{i=1}^{2^n} \sigma_{2i-1} \tilde{f}^{2^{n+1}-2i}$ and thus $$f= \frac{\tilde{f}}{(x_0^2+\dots+x_n^2)^{D-d}}=\frac{\tilde{f}^{2^{n+1}}+\sum_{i=1}^{2^n} \sigma_{2i} \tilde{f}^{2^{n+1}-2i}}{(x_0^2+\dots+x_n^2)^{D-d} \sum_{i=1}^{2^n} \sigma_{2i-1} \tilde{f}^{2^{n+1}-2i}}.$$
We obtain the proposition with $M_1=\sigma_{2^{n+1}}$, $M_2=(x_0^2+\dots+x_n^2)^{D-d}\sigma_{2^{n+1}-1}$, $R_1=\tilde{f}^{2^{n+1}}+\sum_{i=1}^{2^n-1} \sigma_{2i} \tilde{f}^{2^{n+1}-2i}$ and $R_2=(x_0^2+\dots+x_n^2)^{D-d} \sum_{i=1}^{2^n-1} \sigma_{2i-1} \tilde{f}^{2^{n+1}-2i}$.
\end{proof}

\subsection{Linear polynomials}
One can also use Pólya's theorem to show that positive polynomials on simplices one can write with the polynomials which describe the faces. We will use the proof of "`A new bound for Pólya's Theorem with applications to polynomials positive on polyhedra"' \cite{PR} of V.Powers and B.Reznick.

\begin{proposicion}\label{lin}
Let $S_n$ be a $n$-simplex, which we describe with linear polynomials $\lambda_0,\dots,\lambda_n\in \mathbb{R}[X]$ as semialgebraic set. Thus exist for every $f$, which is positive on $S_n$, a finite set of numbers $\left\{a_{\alpha}\in \mathbb{R}^+|\alpha \in \mathbb{N}^{n+1}\right\}$ such that $f=\sum_{\alpha} a_{\alpha}\lambda^{\alpha}$ and thus $f\in T_S$.
\end{proposicion}

\begin{proof}
First we consider the special case $\lambda_0=1-x_1-\dots-x_n, \lambda_1=x_1,\dots, \lambda_n=x_n$ with $f$ positive on the corresponding semialgebraic set and we homogenize $f$ and $\lambda_0$ with a new variable $x_0$ (not to even degree). We imagine in Pólya's theorem the coordinates such that $x_0$ is the factor and $x_0-x_1-\dots-x_n,x_1,\dots,x_n$ are the previous variables. Thus the conditions are exact the same as in Pólya's theorem and we obtain a corresponding equation. If one dehomogenize with $x_0$ one obtain a equation like in the theorem.\\
The general case we obtain with the observation that the theorem is invariant under affine transformations and because every $n$-simplex is an affine transformation of every other $n$-simplex and positive factors in the lambdas do not change the theorem.
\end{proof}

\newpage

\section{Schmüdgen's and Putinar's theorem}
\subsection{Inductive property}

Let $S=\left\{ g_1, \dots, g_s \right\}\subset\mathbb{R}\left[X\right]$ be a finite set of polynomials. We call $S$ or $M_S$ \textit{Putinar}, if every polynomial, which is positive on the semialgebraic set $K_S$, is in the quadratic module $M_S$ defined by $S$. The definition does not depend on $S$ which describes $M_S$ because $K_S=\left\{x \in \mathbb{R}^n | g(x)\geq 0, g\in M_S\right\}$.

\begin{proposicion}\label{induct}\textbf{(Inductive property)}
Let $S=\left\{ g_1, \dots, g_s \right\}\subset\mathbb{R}\left[X\right]$ be a finite set of polynomials such that $K_S$ is compact and furthermore $g_{s+1}$ a polynomial. If $S$ is a set of polynomials which is Putinar, then $S':=S\cup \left\{ g_{s+1} \right\}$ is Putinar. 
\end{proposicion}

\begin{proof} Let $S$ be like above and $f>0$ on $K_{S'}$. It is sufficient to show that exist a $\sigma_{s+1} \in \sum \mathbb{R}[X]^2$ such that $f-\sigma_{s+1} g_{s+1}>0$ on $K_{S}$ because $S$ is Putinar. Let $f$ be such that $f>0$ does not hold for every point in $K_{S}$ (else we choose $\sigma_{s+1}:=0$).\\
Let $M:=\max_{x \in K_{S}\backslash K_{S'}}\left\{\frac{f}{g_{s+1}}(x)\right\}+1>0$, which exists, because the closure of the set $D:=K_{S}\backslash K_{S'}$ is compact, $\frac{f}{g_{s+1}}$ goes to $-\infty$ near of $\bar{D}\backslash D$, when we come from the interior of $D$ (because $f>0$ and $g_{s+1}$ goes to $0$ from negatives) and the function is continuous in $D$.\\ Let $\tilde{\sigma}_{s+1}(x):=\min (M, \frac{f}{2g_{s+1}}(x))$ if $g_{s+1}(x)>0$ and $\tilde{\sigma}_{s+1}(x):=M$ if $g_{s+1}(x)\leq 0$. $\tilde{\sigma}_{s+1}$ is positive on the set $K_{S}$, continuous and $f-\tilde{\sigma}_{s+1} g_{s+1}>0$ on $K_{S}$. With the theorem of Stone-Weierstra\ss\ exist a $r \in \mathbb{R}[X]$ sufficiently close to $\sqrt{\tilde{\sigma}_{s+1}}$, such that $f-r^2 g_{s+1}>0$ on $K_{S}$ (because $g_{s+1}$ is bounded on $K_S$) and thus we have proved the theorem with $\sigma_{s+1}:=r^2$.\\
\end{proof}

\begin{comentario}
The proof shows that one square is sufficient. Analog one can demonstrate the inductive property for every adequate Algebra (of polynomials) and every positive potency (or adequate function), such that to demonstrate a theorem with a more general representation, it is sufficient to show it for one polynomial and apply the inductive property. Wörmann for example showed more general representations, which use potencies of the form $2(2k+1)$ \cite{M1}.
\end{comentario}

\begin{example}\label{esfera}
The set $S=\left\{N-x_1^2-\dots-x_n^2\right\}$ is Putinar for every $N>0$: It is sufficient to show it for the polynomial $1-x_1^2-\dots-x_n^2$ and the general case follows with a dilation. The equation $\frac{1}{2}\left((x_1-1)^2+x_2^2+\dots+x_n^2+(1-x_1^2-\dots-x_n^2)\right)=1-x_1$ shows that $M_S$ contains a tangent hyperplane to the unit sphere and with symmetry every tangent hyperplane. Let $\tilde{S}$ be a set of tangent hyperplanes which forms a simplex. With proposition \ref{lin} we know that $\tilde{S}$ is Putinar and as consequence of proposition \ref{induct}, $\tilde{S}\cup S$ is Putinar. But $\tilde{S}\subset M_S$, thus $M_S=M_{\tilde{S}\cup S}$ and $S$ is Putinar.
\end{example}

\subsection{Putinar's theorem}

\begin{teorema}{\textbf{(Putinar)}}\label{put}\\
Let $S$ be a finite set of polynomials such that $M_S$ contains a polynomial of the form $N-(x_1^2+\dots+x_n^2)$. Then each polynomial $f\in \mathbb{R}[X]$, which is positive on $K_S$, is in $M_S$.
\end{teorema}

\begin{proof}
Example \ref{esfera} and proposition \ref{induct} show the theorem.
\end{proof}

The concept of positivity of polynomials on the infinity appears in the following theorem, such that we will introduce it now. We can consider infinity like $\mathbb{R}P^{n-1}$ and the highest degree parts decide how the polynomials behave on infinity. We have already mentioned, that we need polynomials of even degree in the projective case.\\
For every $p\in \mathbb{R}\left[X\right]$ let $p^g$ be its highest degree part. We call $p$ of even degree positive (non-negative, negative, non-positive) on infinity, if $p^g$ is positive (non-negative, negative, non-positive) on $\mathbb{R}P^{n-1}$. The following propositions show, when we can assure, that such a polynomial like in the theorem of Putinar exists.

\begin{proposicion}\label{n_quads}
Let $S$ be a finite set of polynomials such that $M_S$ contains an element (of even degree), which is negative on infinity. Then $M_{S}$ contains a polynomial of the form $N-x_1^2-\dots-x_n^2$.
\end{proposicion}

\begin{proof} Let $p\in M_S$ be of even degree and negative on infinity and $p^g$ its highest degree part. So $p^g$ is negative except on the origin. With Habicht's theorem we have sums of squares $M_1$, $R_1$, $M_2$ and $R_2$, which are positive except on the origin, such that $M_1$ and $M_2$ are sums of squares of monomials and $-(M_2+R_2) p^g=M_1+R_1$. Thus $(M_2+R_2)p=-M_1-R_1+R \in M_S$ with $R=(M_2+R_2)(p-p^g)$, which has minor degree than $M_1$. $R_1$ is a sum of squares, such that $-M_1+R\in M_S$.\\ When we multiply $-M_1+R$ with a potency of $(x_1^2+\dots+x_n^2)$, we can suppose that $M_1$ has degree $2^d$ with $d\in \mathbb{N}$ and the properties of $M_1$ and $R$ such that $-M_1+R\in M_S$ do not change.\\
If $m$ is a monomial of $-M_1+R$ with degree $d_m$ and $i\in \mathbb{N}$, such that $2^i||d_m$ there exist obviously monomials $m_1$ and $m_2$ of degree $\frac{d_m-2^i}{2}$ and $\frac{d_m+2^i}{2}$ such that $m=2m_1 m_2$. When we add to $-M_1+R$ the square $(m_1+m_2)^2$ we can eliminate $m$. We begin with all monomials with $i=0$ and rise $i$ eliminating the monomials up to the case $i=d-1$. So we can assure that no monomials, with an $i$ minor or equal to the $i$ of $m$, appear. Also the process shoud preserve the property of $M_1$, that it only has monomials which are squares and that each such monomial of its degree appears with positive factor. This is possible when we change $m_1$ and $m_2$ with $\frac{m_1}{\epsilon}$ and $\epsilon m_2$ with $\epsilon>0$ sufficiently small, when $m_2$ has the same degree as $M_1$. We obtain a polynomial of the form $N-M_1$ with $N\in \mathbb{R}$ and $M_1$ of the same form like before.\\
Summing squares of monomials we can eliminate in $N-M_1$ monomials (except $N$), which are no multiples of potencies of one variable and we can achieve that the coefficients of the monomials $x_i^{2^d}$ are equal and negative. Multiplying with a positive real we obtain $N-x_1^{2^d}-\dots-x_n^{2^d}$ (with a different $N$) and summing squares of the form $(x_i^{2^{d-1}}-\frac{1}{2})^2$ we can decrease the exponent to $N-x_1^2-\dots-x_n^2$ (with a different $N$). These transformations do not change, that the polynomial belongs to $M_S$.
\end{proof}

\begin{comentario} There are polynomials, which are non-negative on infinity, but describe a compact set as $-x^4-y^2+1$. In general the problem to show the proposition for every polynomial, which describes a compact set, is as difficult as Hilberts 17th problem, because we have to cancel the highest degree part, but Schmüdgen's theorem shows that it is correct. In general the implication that, if a set is compact that then the quadratic module is Putinar, is not true.
\end{comentario}

\begin{proposicion}\label{ex_neg_inf}
Let $S=\left\{p_1,\dots,p_s\right\}$ be a finite set of polynomials of even degree, such that for each point of the projective space we have a $p_i$, which highest degree part $p^g_i$ is negative on the point. Then $M_{S}$ contains a polynomial negative on infinity.
\end{proposicion}

\begin{proof} The case $n=1$ is trivial.\\
For the general case we consider the highest degree parts $p^g_1,\dots,p^g_s$. We imagine $p^g_i$ as functions on the unit sphere of $\mathbb{R}^n$ with opposite points identified $S^{op}$ (this is well defined because they have even degree) and we point out that in this imagination $x_1^2+\dots+x_n^2=1$. Because of the condition in the proposition, we find non-negative continuous functions $a_1,\dots, a_t$ on $S^{op}$, such that $a_1 p^g_1+\dots + a_t p^g_t<0$ on $S^{op}$ (for example, let $a_i(x)$ be the distance between $x$ and the set $\left\{y|p^g_i(y)\geq 0\right\}$). The theorem of Stone-Weierstra\ss\ applies on the algebra of polynomials, which only has monomials of even degree and the set $S^{op}$. So we find homogeneous polynomials (we homogenize with $x_1^2+\dots+x_n^2$ if necessary) $b_1, \dots, b_t\in \mathbb{R}[x_1,\dots,x_n]$ such that $b_i$ is sufficiently close to $\sqrt{a_i}$ such that $b_1^2 p^g_1+\dots + b_t^2 p^g_t<0$ on $S^{op}$. Multiplying the terms $b_i$ with a potency of $x_1^2+\dots+x_n^2$ we find sums of squares $\sigma_1,\dots,\sigma_n$ such that $\sigma_i p^g_i$ ($i\in \left\{1,\dots,t\right\}$) have the same degree. Because of the homogenity $\sigma_1 p^g_1+\dots + \sigma_t p^g_t<0$ on $\mathbb{R}^n$ except on the origin and so $\sigma_1 p_1+\dots + \sigma_t p_t$ is negative on infinity.
\end{proof}

\begin{comentario}
Polynomials of odd degree can generate polynomials of even degree, which are negative on points of infinity, when we take the product.
\end{comentario}

We will use the generalizations to show a projective version of the theorem of Putinar. The first proof is taken from a paper of Scheiderer \cite{S1} and generalizes Habicht's theorem.

\begin{proposicion}\label{pos_dem}
Let $f\in \mathbb{R}[X_0]$ be a homogeneous polynomial of even degree, positive except on the origin. Thus exist a $N\in \mathbb{N}$ such that $(x_0^2+\dots+x_n^2)^N f$ is a sum of squares of polynomials.
\end{proposicion}

\begin{proof}
$f$ is positive on the unit sphere, which we describe with $g:=-(1-(x_0^2+\dots+x_n^2))^2$, which is negative on infinity. Thus, because of Putinar's theorem, exist sums of squares of polynomials $\sigma_0$ and $\sigma_1$ such that $f=\sigma_0+g \sigma_1$.\\ We homogenize with $z$ to even degree and we write the squares of $\sigma_0$ and $\sigma_1$ in the form $(f_1+z\cdot f_2)^2=f_1^2+2f_1zf_2+z^2f_2^2$, such that in $f_1$ and $f_2$ the variable $z$ only appears in even potency. Thus, we see, that we can write the right-hand side with sum of squares, where $z$ only appears in even potency such that the terms $2f_1zf_2$ have to disappear.\\ When we write $x_1^2+\dots+x_n^2$ instead of $z^2$ we obtain a sum of squares $\tilde{\sigma}_0$, such that $(x_0^2+\dots+x_n^2)^N f= \tilde{\sigma}_0$ and so it is a homogeneous identity.
\end{proof}

\begin{proposicion}\label{put_proy1}\textbf{(Projective version)(special case)}\\
Let $S:=\left\{g_1, \dots, g_s\right\}\subset \mathbb{R}[X_0]$ be homogeneous polynomials of even degree and we consider the corresponding semialgebraic set $K_S^p$. Let $f$ be a homogeneous polynomial of even degree, which is positive on $K_S^p$. Thus exist a $N\in \mathbb{N}$, such that $(x_0^2+\dots+x_n^2)^N f$ is in the quadratic module $M_S^h$. 
\end{proposicion}

\begin{proof}
The proposition follows from the inductive propery and the prior proposition. The demonstration of the inductive property is analog to the proof in the affine case, when we use the unit sphere with opposite points identificated as compact set and the algebra of polynomials which only have monomials of even degree for the theorem of Stone-Weierstra\ss\ and when we homogenize with $x_0^2+\dots+x_n^2$, when it is necessary.
\end{proof}

\begin{comentario}
We will see a generalization to positive denominators at the end of this part. We pay attention, that this projective version does not need a condition as existence of polynomials negative on infinity.
\end{comentario}

\subsection{Schmüdgen's theorem}

To demonstrate Schmüdgen's theorem, we need stronger tools, which one can follow for example from Tarski's transfer principle. We will use the Positivstellensatz of Krivine and Stengle \cite{M1} to show Schmüdgen's theorem with a similar step like in literature, but less abstract (compare \ref{neg_inf} and \ref{descenso}).

\begin{teorema}{\textbf{(Krivine, Stengle)}}\\
Let $S$ be a finite set of polynomials of $\mathbb{R}[X]$, $K_S$ and $T_S$ the corresponding semialgebraic set and the corresponding preorden. Then:
\begin{enumerate}
\item $f>0$ on $K_S$ $\Longleftrightarrow$ there are $p,q\in T_S$ such that $pf=1+q$
\item $f\geq 0$ on $K_S$ $\Longleftrightarrow$ there are $m\in \mathbb{N}$ and $p,q\in T_S$ such that $pf=f^{2m}+q$
\item $f=0$ on $K_S$ $\Longleftrightarrow$ there is a $m\in \mathbb{N}$ such that $-f^{2m}\in T_S$.
\end{enumerate}
\end{teorema}

\begin{proof}
One can show in a elementary way, that these three parts are equivalent. A proof one can find for example in \cite{M1}.
\end{proof}

\begin{proposicion}\label{neg_inf}
Let $S$ be a finite set of polynomials, such that $K_S$ is compact. Then $T_S$ contains a polynomial, which is negative on infinity.
\end{proposicion}

\begin{proof}
Given that $K_S$ is compact, we find a sufficiently big $N\in \mathbb{N}$, such that $N-(x_1^2+\dots+x_n^2)$ is positive on $K_S$. With the Positivstellensatz of Krivine and Stengle we find $p,q\in T_S$, such that $p(N-(x_1^2+\dots+x_n^2))=1+q$. So $Np-1=q+p(x_1^2+\dots+x_n^2)\in T_S$.\\ We write $y$ instead of $\frac{x_1^2+\dots+x_n^2}{N}$ and we obtain $Np-1=q+Npy$. Adding $-y-y^2-\dots-y^{l+1}$ to the equation, we obtain $Np-1-y-\dots-y^{l+1}=q+y(Np-1-y-\dots-y^l)$.\\ The polynomials $q$, $y$ and $Np-1$ ($l=0$) are elements of $T_S$ and the equation is a step in an induction to prove that $Np-1-y-\dots-y^{l+1}\in T_S$. So it is true for every $l\in \mathbb{N}$. With $l$ sufficiently big, such that $p$ has a minor degree than $2l$, we obtain a polynomial, which is negative on infinity.
\end{proof}

\begin{comentario} To demonstrate that $T_S$ contains also a polynomial of the form $N-x_1^2-\dots-x_n^2$, we do not need Habicht's theorem because our highest degree part is in the form we need, to follow the demonstration in that proof.
\end{comentario}

\begin{teorema}{\textbf{(Schmüdgen)}}\\
Let $S$ be a finite set of polynomials, such that $K_S$ is compact. Thus, each polynomial $f\in \mathbb{R}[X]$, which is positive on $K_S$, is in $T_S$.
\end{teorema}

\begin{proof}
The propositions \ref{neg_inf}, \ref{n_quads} show that $T_S$ contains a polynomial of the form $N-(x_1^2+\dots+x_n^2)$. We obtain the result with Putinar's theorem \ref{put}.
\end{proof}

We will show a theorem, which generalizes Schmüdgen's theorem to the homogeneous case and so we obtain a very general Positivstellensatz in projective language. Therefore we need a homogeneous version of Zeng \cite{Z} of the Positivstellensatz of Krivine and Stengle.

\begin{teorema}{\textbf{(Zeng)(special case)}}\\
Let $S, T$ be finite sets of homogeneous polynomials of $\mathbb{R}[X_0]$. Let $T_{S\cup T}^h$ be the homogeneous preorden generated by $S\cup T$ and $P$ the set of products (semigroup) of elements of $T$. Let $f$ be homogeneous and positive on $K_S\cap \left\{X_0| \forall g\in T: g(X_0)>0\right\}$. Then there are $p,q\in T_{S\cup T}^h$ and $t\in P$ such that $pf=t+q$ as homogeneous identity.
\end{teorema}

\begin{proof}
\cite{Z}
\end{proof}

\begin{teorema}\label{schm_hom}{\textbf{(Homogeneous version of Schmüdgen's theorem)}}\\
Let $S$ be a finite set of homogeneous polynomials, $K_S$ its corresponding semialgebraic set and $K_S=K\dot{\cup}K_0$ a disjoint union of closed sets, such that $\mathring{K_S}\subseteq K$. Let $g$ be a homogeneous polynomial of degree $d>0$, which is positive on $K$ and which equals zero on $K_0$ (so the closure of $K_0$ in Zariski's topology does not cut $K$).\\ Then for every homogeneous polynomial $f$, which has a degree which is a multiple of $d$ and which is positive on $K$, exist a natural number $N$, such that $g^N f$ is in the homogeneous preorden $T_S^h$.
\end{teorema}

We will show it in the same manner like Marshall \cite{M1} shows Schmüdgen's theorem in his book (it goes back to Wörmann), but we generalize the concept of the archimedean property.

\begin{defi}
We call a homogeneous quadratic module $M$ archimedean with respect to $g\in M$, such that $gM\subseteq M$, if for every homogeneous polynomial $f$, which has a degree, which is a multiple of the degree of $g$ (i.e. $\deg f= l\cdot \deg g$) exist $N, k \in \mathbb{N}_0$, such that $g^{k+l}N\pm g^k f\in M$.
\end{defi}

\begin{proposicion}
Let $M$ be a homogeneous quadratic module and $g\in M$ of degree $d$ such that $gM\subseteq M$. If $M$ contains a polynomial of the form $g^{k}(g^2 N- (\sum x_i^2)^{d})$, then it is archimedean with respect to $g$. 
\end{proposicion}

\begin{proof}
Let $A$ be the subset of the polynomials $f$, which fulfill the generalized archimedean property with respect to $g$. $A$ contains the polynomials of degree $0$ because the reals are archimedean.\\
For every monomial $m$ of degree $d$ we have $g^{k}(g^2 N-m^2)\in M$ (and obviously $g^{k}(g^2 N+m^2)\in M$), because in $g^{k}(g^2 N- (\sum x_i^2)^{d})$ appears every $m^2$ in the subtrahend with factor $\leq -1$, such that we only have to add $g^k((\sum x_i^2)^{d}-m^2)\in M$.\\
With the identity: $g^{k+1}\left(g N \pm m\right)=\frac{1}{2}\left(g^{k+2}(N-1)+g^k(g^2 N-m^2)+g^k(m\pm g)^2\right)\in M$ we know, that the monomials of degree $d$ are in $A$. $A$ is closed under sums of elements of the same degree $ld$ with $l\in \mathbb{N}$ because $g^{k_1+k_2}\left(g^l (N_1+N_2)\pm (a_1+a_2)\right)=g^{k_2}g^{k_1}(g^l N_1\pm a_1)+g^{k_1}g^{k_2}(g^l N_2\pm a_2)\in M$, when $g^{k_1}(g^l N_1\pm a_1)\in M$ and $g^{k_2}(g^l N_2\pm a_2)\in M$.\\
The identity $ab=\frac{1}{4}\left((a+b)^2-(a-b)^2\right)$ (together with the fact that we can homogenize with $g$) shows, that, to show that $A$ is closed under products, it is sufficient to show it for a square. This follows from the identity:\\ $g^{k+l} \left(g^{2l} N^2-a^2\right)=\frac{1}{2N}\left((g^l N+a)^2(g^{k+l} N-g^k a)+(g^l N-a)^2(g^{k+l} N+g^k a)\right)\in M$ when $g^{k+l} N\pm g^k a\in M$ (and obviously $g^{k+l}(g^{2l}N^2+a^2)\in M$).
\end{proof}

\begin{proposicion}
Under the conditions of \ref{schm_hom} with $g\in T_S^h$ follows, that $T_S^h$ contains a polynomial of the form $g^{k}(g^2 N- (\sum x_i^2)^{d})$. 
\end{proposicion}

\begin{proof}
We chose $N_0$ sufficiently big, such that $g^2 N_0- (\sum x_i^2)^{d}$ is positive on $K$. With the homogeneous Positivstellensatz ($S$ as in \ref{schm_hom} and $T=\left\{g\right\}$) follows $p(g^2 N_0- (\sum x_i^2)^{d})=g^{m}+q$ with $p,q\in T_S^h$ and so $(g^{m}+q)(g^2 N_0- (\sum x_i^2)^{d})=p(g^2 N_0- (\sum x_i^2)^{d})^2\in T_S^h$.\\
Let $T'$ be the homogeneous preorden generated by $T_S^h$ and $g^2 N_0- (\sum x_i^2)^{d}$. $T'$ is archimedean, because of the prior proposition. So we have $N_1, k_1\in \mathbb{N}$, such that $g^{k_1}(g^{m}N_1-q) \in T'$.\\ 
Multiplying with $g^{m}+q\in T_S^h$ we obtain $g^{k_1}(g^{m} N_1-q)(g^{m}+q)\in T_S^h$ (because the generator $g^2 N_0- (\sum x_i^2)^{d}$ of $T'$ multiplied with $g^{m}+q$ is in $T_S^h$).\\
When we add $g^{k_1}\left(g^{m}\frac{N_1}{2}-q\right)^2\in T_S^h$, we obtain $g^{k_1+m}\left(g^{m}\left(N_1+\frac{N_1^2}{4}\right)-q\right)\in T_S^h$.\\
When we multiply with $N_0$ and add $g^{k_1+m-2}(g^{m}+q)(g^2 N_0- (\sum x_i^2)^{d})\in T_S^h$ and $g^{k_1+m-2}q (\sum x_i^2)^{d}\in T_S^h$ we obtain: $g^{k_1+2m-2}\left(g^2N_0\left(\frac{N_1}{2}+1\right)^2-(\sum x_i^2)^{d}\right)\in T_S^h$.\\ With $k=k_1+2m-1$ and $N=N_0\left(\frac{N_1}{2}+1\right)^2$, we obtain our proposition.
\end{proof}

\begin{proposicion}\label{descenso}
If in the conditions of \ref{schm_hom} $T:=T_S^h$ is archimedean with respect to $g\in T_S^h$, then for each homogeneous $f$, which has an degree which is a multiple of the degree of $g$ (i.e. $\deg f= l\cdot \deg g$) and which is positive on $K$, exist a $N\in \mathbb{N}$, such that $g^{N} f$ is in the homogeneous preorden $T_S^h$.
\end{proposicion}

\begin{proof}
Because of the homogeneous Positivstellensatz of Zeng exist $m \in \mathbb{N}$ and $s,t\in T$, such that $ft=g^{l+m}+s$ is a homogeneous identity and so $ft-g^{l+m}\in T$ ($\deg t=m \deg g$).\\ Let $\Sigma =\left\{r\in \mathbb{Q}| \exists k\in \mathbb{N}: g^{k+l}r+g^{k} f\in T \right\}$, which is not empty, because $T$ is archimedean with respect to $g$. Let $r_1\in \Sigma, r_1\geq 0$ and we choose for $r_i\in \Sigma, r_i\geq 0, i>0$ (which we choose later) representations: $g^{k_i+l}r_i+g^{k_i} f\in T$. Furthermore, let $\kappa ,k_0\in \mathbb{N}$, such that $g^{k_0+m}\kappa -g^{k_0} t\in T$. Then
\begin{align*}
&\kappa g^{k_i+k_0+m}\left(g^{l}(r_i-\frac{1}{\kappa })+f\right)=g^{k_i+k_0+m}\left(g^{l}(\kappa  r_i-1)+\kappa f\right)\\
&=(g^{k_0+m}\kappa -g^{k_0} t)(g^{k_i+l}r_i+g^{k_i} f)+g^{k_i+k_0}(ft-g^{l+m})+g^{k_i+k_0+l}r_i t\in T
\end{align*}
such that $r_{i+1}:=r_i-\frac{1}{\kappa }\in \Sigma$. We repeat this process until we obtain a $r_i$ less than $0$ and so there is a $N\in \mathbb{N}$, such that $g^{N} f\in T_S^h$, because we can eliminate potencies of $g$ with negative factor.
\end{proof}

\begin{proof}[\textbf{Proof of \ref{schm_hom}}]
The propositions show the theorem for the polynomial $g^2\in T_S^h$. The theorem for $g$ we obtain, when we apply the theorem with $g^2$ to $f$ or $gf$, depending on $f$ (if its degree is an even or odd multiple of the degree of $g$).
\end{proof}

\begin{example}
Special cases of this homogeneous version (\ref{schm_hom}):
\begin{itemize}
\item Let $S=\left\{(x_2^2-2x_1^2-2x_0^2)(x_2^2-x_1^2-x_0^2)^2\right\}$ and $g=x_2^2-x_1^2-x_0^2$. So for each homogeneous polynomial of even degree $f$, which is positive on $K:=\left\{X_0|x_2^2-2x_1^2-2x_0^2\geq 0\right\}$, exist a $N\in \mathbb{N}$, such that $g^N f=\sigma_0+\sigma_1 (x_2^2-2x_1^2-2x_0^2)(x_2^2-x_1^2-x_0^2)^2$ with $\sigma_0, \sigma_1 \in \mathbb{R}[X_0]^2_h$.
\item Schmüdgen's theorem: $g=x_0$, $K$ is compact in affine space, $K_0=\left\{X_0|x_0=0\right\}$, $S$ is homogenized with $x_0$ and contains furthermore $x_0$
\item Corollary 4.5. of \cite{S4}: We introduce a new variable $x_{-1}$ in the problem without changing $S$ or $f$, which shall be polynomials of even degree and we take $K=K_S$. We choose a corresponding $g$ and to obtain each $f$ of even degree, we homogenize $f$ with $x_{-1}$ to an adequate degree. After applying our theorem, we dehomogenize with $x_{-1}$ and we obtain a representation of the homogeneous polynomial $g^N f$ in the non-homogeneous preorden of $S$.
\item Generalization of \ref{pos_dem}: $K$ is whole space, $g$ is positive, $f$ is positive with an degree which is a multiple of the degree of $g$, thus $g^N f$ is a sum of squares for a sufficiently big $N$
\item When $K_S$ does not have an interior point, we can represent each homogeneous polynomial $f$ of adequate degree in the form $g^N f\in T_S^h$, when $g$ is zero on $K_S=K_0$. So when $g$ is homogeneous and non-negative and $S=\left\{-g\right\}$ we obtain for each homogeneous $f$ of adequate degree a $N\in \mathbb{N}$ and homogeneous sums of squares $\sigma_0$ y $\sigma_1$, such that $g=\frac{\sigma_0-fg^N}{\sigma_1}$.
\end{itemize}
\end{example}

\begin{comentario}\label{open} \begin{itemize}
\item With theorem 4.3.6. of \cite{M1} we have a different generalization of Schmüdgen's theorem. It motivates to generalize the archimedean property another time. This could be content of a future work.
\item To obtain representation for each degree, we can use sometimes irreducible components of $g$, such that the denominator is a product of the irreducible components of $g$. When it is not possible because of degrees or positivity, we can use sum of squares in the case of even degrees.
\end{itemize}
\end{comentario}

We will use these results to obtain a projective version of Putinar's theorem:

\begin{teorema}\label{put_hom}\textbf{(Projective version of Putinar's theorem)}
Let the conditions be as in \ref{schm_hom}, such that each polynomial of $S$ and $g$ (and so $f$) have even degree. If there exist a polynomial $g_- \in M_S^h$, such that $g_-(X_0)<0$, when $g(X_0)\leq 0$ and such that $K\subseteq \left\{X_0|g_-(X_0)>0\right\}$, then for each polynomial $f$, which is positive on $K$ and which has an degree which is a multiple of $d$, exist a $N\in \mathbb{N}$ with $g^N f\in M_S^h$.
\end{teorema}

\begin{proof}
We use a inductive property as in \ref{induct} and \ref{put_proy1}, to reduce the problem to the case of the polynomial $p_-$. The case of one polynomial follows from Schmüdgen's theorem, because $M_S^h$ and $T_S^h$ are identical in this case.\\ 
For the inductive property, we consider the set $\left\{X_0| g(X_0)=1, g_-(X_0)\geq 0\right\}$, which is bounded in $\mathbb{R}^{n+1}$, because if there would be a sequence with arbitrary big distance of $X_0$ to the origin, then the projection induces a sequence on the unit sphere, which has, with the compactness, a convergent subsequence with a limit $P$ with $g(P)=0$ and $g_-(P)\geq 0$, which is not possible. So we find a sufficiently big $M\in \mathbb{N}$, such that $K_{g_-}:=\left\{X_0|g(X_0)=1, \sum x_i^2\leq M\right\}\supseteq \left\{X_0| g(X_0)=1, g_-(X_0)\geq 0\right\}$. We use $K_{g_-}$ with opposite points identified and the algebra of polynomials which only have monomials of degrees, which are multiples of $d$, for the theorem of Stone-Weierstra\ss\ .\\
Before applying the inductive property, we multiply the $g_i$ with a potency of $x_0^2+\dots+x_n^2$ to a degree, which is a multiple of $2d$ and so we do not change the properties of the $g_i$ and we can also suppose that $f$ has a degree which is a multiple of $2d$ (if it is not we can multiply $f$ with $g$) such that we can homogenize later without problems. As in \ref{induct} we obtain $(g)f-\sigma_i g_i$, which is positive on the bigger semialgebraic set. Now we homogenize with $g$ as in \ref{put_proy1} and we obtain homogeneous polynomials which are also positive in the bigger projective semialgebraic set.
\end{proof}

\begin{comentario}
When $g$ is positive we obtain a generalization of \ref{put_proy1} because we can choose $g_-$ as $1$.
\end{comentario}

Like in the affine theorem of Putinar, we can find weaker properties, which ensure that there is such a polynomial $g_-$ like in the quadratic module.

\begin{proposicion}
When, in the conditions of \ref{schm_hom} with polynomials of even degree, for each point $P$ of $K_0\subset \mathbb{R}P^n$ exist a $g_i\in S$, such that $g_i(P)<0$, then there is a polynomial $p_-\in M_S^h$ as in the last theorem.
\end{proposicion}

\begin{proof}
The proof is as in \ref{ex_neg_inf} with the theorem of Stone-Weierstra\ss, but we use it only for the subset $\left\{X_0|g(X_0)\leq 0\right\}$ of the projective space, such that we obtain a polynomial $\tilde{g}_-\in M_S^h$ negative on $\left\{X_0|g(X_0)\leq0\right\}$ (the condition in the theorem assures, that for each $P$ with $g(P)\leq 0$ there is a $g_i\in S$ with $g_i(P)<0$). To obtain, that $K\subseteq \left\{X_0|g_-(X_0)>0\right\}$, we see that $\tilde{g}_-\in M_S^h$ and so $\tilde{g}_- \geq 0$ on $K$, such that, if $2k$ is the degree of $\tilde{g}_-$, we only have to choose a sufficiently small $\epsilon>0$ to obtain $g_-:=\tilde{g}_- + \epsilon (x_0^2+\dots+x_n^2)^k<0$ on $\left\{X_0|g(X_0)\leq0\right\}$.
\end{proof}

\newpage

\section{Non-compact case}

\subsection{Case $n=1$ and $n\geq 3$}

Questions about sets in $K[x]$ and sets with dimension $3$ or higher are solved. We will only mention the results.

\begin{proposicion}[$n=1$, \cite{KM}]
Let $S$ be a finite set of polynomials, such that $K_S$ is not compact and we normalize the elements of $S$, such that the coefficients of their degree equal $-1$ or $1$. We describe a set of natural generators $T$ of $K_S$:
\begin{itemize}
\item When $K_S$ has a minimum $a$ then $x-a \in T$
\item When $K_S$ has a maximum $b$ then $b-x \in T$
\item When $K_S$ contains $a$ and $b$, such that $K_S\cap (a,b)=\emptyset$, so $(x-a)(x-b) \in T$
\item $T$ does not contain more elements which are necessary of the anterior conditions.
\end{itemize}
So $S$ is Putinar if and only if $T\subseteq S$.
\end{proposicion}

\begin{proof}
\cite{KM}
\end{proof}

\begin{example}
\begin{itemize}
\item Let $K_S=\left[0,1\right]\cup\left[2,\infty\right)$. So the natural generators of $K_S$ are $\left\{x, (x-1)(x-2)\right\}$.
\item Let $S=\left\{x\right\}$ and $S'=\left\{x, (x-1)^3\right\}$. So $S$ is Putinar and $S'$ is not Putinar and so the inductive property does not hold for non-compact sets.
\end{itemize}
\end{example}

\begin{proposicion}[$n\geq 3$, \cite{S2}]
Let $S$ be a finite set of polynomials, such that $K_S$ is non-compact and has dimension $3$ or higher. So exist a polynomial $f$ which is positive on $K_S$ and $f\notin T_S$.
\end{proposicion}
\begin{proof}
\cite{S2}
\end{proof}

\begin{comentario}
Here the dimension is a condition for the set $K_S$. We claim, that the dimension shall be higher than $2$ for arbitrary large distances of the origin.
\end{comentario}

\subsection{Stability}

Schmüdgen's theorem holds, if the set is compact or in projective language $K$ and $K_0=\left\{X_0|x_0=0\right\}$ do not have common points. In the non-compact case we consider the question, when we have common points, lines or sets of higher dimension. If the set is sufficiently big the denominators, which are potencies of $x_0$ have to appear also in the sums of squares as factor, if we can "`measure"' the multiplicity of $x_0$ on the common set. In affine language this means, that we have a relation between $f$ and the summands in the representation, which has the name stability.

\begin{defi}
A quadratic module $M_S$ is stable, if for every $d$ we find a $d'$, such that for each $p\in M_S$ it has a presentation in $M_S$ with $deg(\sigma_i p_i)\leq d'$. 
\end{defi}

Scheiderer \cite{S3} has proved that quadratic modules, which are not stable, do not have SMP and so affine representations without denominators do not exist. So we ask, when quadratic modules are stable, to obtain negative results. In a paper of 2009 Netzer \cite{N} proved that some sets of tentacles are stable with the use of different graduations. The graduations are important, because they change the appearance of the infinity.

\begin{defi}
Let $S$ be a finite set of polynomials. We say that $K_S$ contains a standard tentacle $T_{K,z}:=\left\{(\lambda^{z_1}x_1,\dots,\lambda^{z_n}x_n)|\lambda\geq 1, (x_1,\dots,x_n)\in K\right\}$ for a compact set with an interior point $K$ and direction $z=(z_1,\dots,z_n)$, when $T_{K,z}\subseteq K_S$.
\end{defi}

\begin{teorema}
When $K_S$ contains standard tentacles with directions $z^{(1)},\dots,z^{(m)}$, such that exist $r_1,\dots, r_m\in \mathbb{N}$ with $r_1 z^{(1)}+\dots+r_m z^{(m)}\succ 0$, then $M_S$ is stable. These numbers exist if and only if there are no bounded polynomials on $K_S$ except constants.\\
The relation $\succ$ signifies that each component of the vector on the left side is bigger than the corresponding component on the right side.
\end{teorema}

\begin{proof}
We will show only the first part for $n=2$ with one or two tentacles (when we have (in the case $n=2$) more than $2$ tentacles like in the theorem, we have also $2$ tentacles with the property). A complete proof you can find in \cite{N} or \cite{NM}.\\
Let $p\in M_S$ be of degree $d>0$ and $\sigma_i q_i$ a summand in a representation of $p$ with degree $d'>0$.\\ First we suppose, that we have one tentacle $T_{K,z}$ with $z_1>0$ and $z_2>0$ and without loss of generality $z_1\geq z_2$. We substitute $x_0 \lambda^{z_1}$ for $x$ and $y_0\lambda^{z_2}$ for $y$ and we consider $s(\lambda):=p$, $t(\lambda):=\sigma_i q_i$ as polynomials in $\lambda$ with coefficients in $x_0$ and $y_0$. With this substitution it is not possible, that some parts cancel, because we did not substitute special values in $x_0$ and $y_0$ such that $d'\cdot z_2=deg(\sigma_i q_i)\cdot z_2\leq deg_{\lambda}(t)$. The coefficient of the highest degree of $t$ does not equal zero in $K$ because $K$ contains a interior point. So we can choose special values for $x_0$ and $y_0$ without reducing the degree of $t$. This monomial decides the sign of $t$, when $\lambda$ goes to infinity and so it is positive, because the corresponding trajectory is in $T_{K,z}$. The summands $\sigma_j g_j$ are non-negative on the trajectory, so we have $deg_{\lambda}(t)\leq deg_{\lambda}(s)$. Thus, with $d\cdot z_1=deg(p)\cdot z_1 \geq deg_{\lambda}(s)$ we obtain $d'\leq d\cdot \frac{z_1}{z_2}$.\\
In the case of two tentacles we have $z^{(1)}=(z_1^{(1)},z_2^{(1)})$ and $z^{(2)}=(z_1^{(2)},z_2^{(2)})$ and without loss of generality $z_1^{(1)}>0$, $z_2^{(1)}\leq 0$, $z_2^{(2)}>0$, $z_1^{(2)}\leq 0$ and $r_1, r_2 \in \mathbb{N}$ with $r_1 z^{(1)}+r_2 z^{(2)}\succ 0$. We follow the case of one variable in both tentacles and we obtain from $p$ and $\sigma_i q_i$ the polynomials $s_1$, $s_2$, $t_1$ and $t_2$, which correspond to the tentacles like above. So we have $deg_{\lambda}(s_1)\geq deg_{\lambda}(t_1)$, $deg_{\lambda}(s_2)\geq deg_{\lambda}(t_2)$ (which can be negative) and so $r_1 deg_{\lambda}(s_1)+r_2 deg_{\lambda}(s_2)\geq r_1 deg_{\lambda}(t_1)+r_2 deg_{\lambda}(t_2)$. Furthermore, we have $z_1^{(1)} d\geq deg_{\lambda}(s_1)$ and $z_2^{(2)} d\geq deg_{\lambda}(s_2)$ and thus $(r_1 z_1^{(1)} + r_2 z_2^{(2)})d \geq r_1 deg_{\lambda}(s_1)+r_2 deg_{\lambda}(s_2)$. We also have $r_1 deg_{\lambda}(t_1)+r_2 deg_{\lambda}(t_2)\geq \min(r_1 z_1^{(1)}+r_2 z_1^{(2)}, r_1 z_2^{(1)}+r_2 z_2^{(2)}) d'$, because the right-hand side is the minimum $\lambda$-degree of $q^{r_1}$ with $z^{(1)}$-grading times $q^{r_2}$ with $z^{(2)}$-grading, for any polynomial $q$ of degree $d'$ where highest degree parts are not cancelling out and the left-hand side is the case $q=\sigma_i q_i$. When we combine the results, we also obtain a linear inequality for $d'$ dependent on $d$ and constants.
\end{proof}

\begin{example}
To illustrate the result we present examples, which come from \cite{N}:
\begin{center}
\begin{tabular}{ll}
$0.5x\leq y\leq x, x\geq 0$ \hspace{1cm} & $0.5 x^2\leq y\leq x^2, x\geq 0$\\
tentacle $(1,1)\succ 0$ & tentacle $(2,1)\succ 0$\\
& \\
\begin{tikzpicture}[scale=0.5]
\draw[thick] (-3,0) -- (3,0);
\draw[thick] (0,-2) -- (0,2);
\draw[color=black, domain=-2:2] plot (\x, \x);
\draw[color=black, domain=-1.5:1.5] plot (2*\x, \x);
\draw[pattern=north west lines] (0,0) --(2,2) -- (3,2) -- (3,1.5);
\end{tikzpicture} &
\begin{tikzpicture}[scale=0.5]
\draw[thick] (-3,0) -- (3,0);
\draw[thick] (0,-2) -- (0,2);
\draw[color=black, domain=-1.41:0] plot (\x, -\x^2);
\draw[color=black, domain=0:1.41] plot (\x, \x^2);
\draw[color=black, domain=-2:0] plot (\x,0.5*-\x^2);
\draw[color=black, domain=0:2] plot (\x,0.5*\x^2);
\draw[pattern=north west lines] plot[smooth,domain=0:1.41](\x, \x^2) --plot[smooth,domain=2:0](\x, 0.5*\x^2);
\end{tikzpicture}
\end{tabular}
\end{center}

\begin{center}
\begin{tabular}{ccc}
$(x-1)(y-1)\leq 1, x\geq 0, y\geq 0$& $(x-1)y\leq 1, x\geq 0, y\geq 0$& $x^2y\leq 1, xy\geq -1, x\geq 0, (1-x)y\geq 0$\\
tentacles $(1,0), (0,1)$ & tentacles $(0,1), (1,-1)$  & tentacles $(-1,2), (1,-1)$  \\
$(1,0)+(0,1)=(1,1)\succ 0$ & $2\cdot (0,1)+(1,-1)=(1,1)\succ 0$  & $2\cdot (-1,2)+3\cdot(1,-1)=(1,1)\succ 0$  \\
& & \\
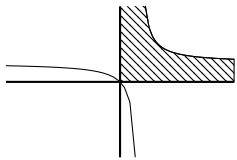
\begin{tikzpicture}[scale=0.25]
\draw[thick] (-6,0) -- (6,0);
\draw[thick] (0,-4) -- (0,4);
\draw[color=black, domain=0.33:5] plot (1+\x,1+\x^-1);
\draw[color=black, domain=-7:-0.2] plot (1+\x,1+\x^-1);
\draw[pattern=north west lines] plot[smooth,domain=0.33:5](1+\x, \x^-1 +1) -- (6,0) -- (0,0) -- (0,4);
\end{tikzpicture}&
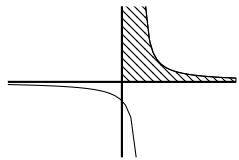
\begin{tikzpicture}[scale=0.25]
\draw[thick] (-6,0) -- (6,0);
\draw[thick] (0,-4) -- (0,4);
\draw[color=black, domain=0.25:5] plot (1+\x,\x^-1);
\draw[color=black, domain=-7:-0.25] plot (1+\x,\x^-1);
\draw[pattern=north west lines] plot[smooth,domain=0.25:5](1+\x, \x^-1) -- (6,0) -- (0,0) -- (0,4);
\end{tikzpicture}&
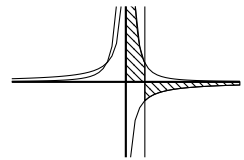
\begin{tikzpicture}[scale=0.25]
\draw[thick] (-6,0) -- (6,0);
\draw[thick] (0,-4) -- (0,4);
\draw[color=black, domain=-6:-0.5] plot (\x,-\x^-2);
\draw[color=black, domain=0.5:6] plot (\x,\x^-2);
\draw[color=black, domain=-6:-0.25] plot (\x,-\x^-1);
\draw[color=black, domain=0.25:6] plot (\x,-\x^-1);
\draw[color=black, domain=-4:4] plot (1,\x);
\draw[pattern=north west lines] plot[smooth,domain=0.5:1](\x, \x^-2) -- (1,0) -- (0,0) -- (0,4);
\draw[pattern=north west lines] plot[smooth,domain=1:6](\x, -\x^-1) -- (6,0) -- (1,0) -- (1,-1);
\end{tikzpicture}
\end{tabular}
\end{center}
\end{example}

There are more propositions for SMP, which imply a negative answer for the question of positive polynomials, when it is negative for SMP. For example, we can mention a theorem of Schmüdgen of 2003, which shows that one can reduce the question of SMP to fibres of bounded polynomials. In \cite{NM} are also new results, which consider more general tentacles. 

\subsection{Cylinders}

In 2008 Marshall \cite{M2} proved, that each polynomial, which is non-negative on the strip $[0,1]\times \mathbb{R}$, is in the quadratic module generated by $x(1-x)$. The tools are very technical. Nguyen and Powers \cite{NP} generalize these results of Marshall and show a theorem when we change the interval $[0,1]$ with a compact set of dimension $1$ with natural generators:

\begin{teorema}\label{cil}
Let $U$ be a compact set of $\mathbb{R}$ and $S$ be the set of its natural generators. Thus, each polynomial $f$, which is non-negative on $U\times \mathbb{R}$, is in the quadratic module generated by $S$.
\end{teorema}

Furthermore, Nguyen and Powers \cite{NP} use three techniques to generalize the result. We will explain two techniques and we will present them a bit more general, to obtain more examples, but we use the same proof.

\begin{proposicion}{\textbf{(Elimination of squares)}}\\
Let $U$ be a subset of $\mathbb{R}^2$, which is symmetric with respect to the $x$-axis, which has a finite set of generators $S$ of the form $g_i(x,y^2)$ and which has the property that each positive (non-negative) polynomial on $U$ is in the cuadratic module $M_S$. Thus, each polynomial positive (non-negative) on $U^+:=K_{S'}$, which we describe with $S':=\left\{g_i(x,y)\right\}\cup \left\{y\right\}$, is in $T_{S'}$.
\end{proposicion}

\begin{proof}
Let $f$ be positive (non-negative) on $U^+$. So $f(x,y^2)> 0$ ($\geq 0$) on $U$ and $f(x,y^2)\in M_S$. So $f(x,y^2)=\sum_i{\sum_j s_{i,j}^2 g_i(x,y^2)}$ ($g_0:=1$). Thus, $$f(x,y^2)=\frac{f(x,y^2)+f(x,(-y)^2)}{2}=\frac{\sum_i{\sum_j s_{i,j}(x,y)^2 g_i(x,y^2)}+\sum_i{\sum_j s_{i,j}(x,-y)^2 g_i(x,y^2)}}{2}$$
$$=\frac{\sum_i{\left(\sum_j{ s_{i,j}(x,y)^2 + s_{i,j}(x,-y)^2}\right) g_i(x,y^2)}}{2}.$$
When we write $s_{i,j}(x,y)=\hat{s}_{i,j}(x,y^2)+y\check{s}_{i,j}(x,y^2)$ we have $s_{i,j}(x,y)^2 + s_{i,j}(x,-y)^2=\hat{s}_{i,j}(x,y^2)^2+y^2\check{s}_{i,j}(x,y^2)^2$.
Thus $f(x,y^2)=\frac{1}{2}\sum_i{\left(\sum_j{ \hat{s}_{i,j}(x,y^2)^2 + y^2\check{s}_{i,j}(x,y^2)}\right) g_i(x,y^2)}$ and \\
$f(x,y)=\frac{1}{2}\sum_i{\left(\sum_j{ \hat{s}_{i,j}(x,y)^2 + y\check{s}_{i,j}(x,y)}\right) g_i(x,y)}$.
\end{proof}

\begin{proposicion}{\textbf{(Automorphisms of $\mathbb{R}[x,y]$)}}\\
The automorphisms of $\mathbb{R}[x,y]$ are generated by a finite number of affine transformations and transformations of the form $x\mapsto x, y\mapsto y+q(x)$ with $q(x)\in \mathbb{R}[x]$.
\end{proposicion}

\begin{proof}
\cite{MW}
\end{proof}

Automorphisms are compatible with the property of positivity and so this proposition says, that we can use these substitutions to each set to obtain new examples. So we obtain:

\begin{teorema}
Let $U$ be a compact set of $\mathbb{R}$, $S$ the set of its natural generators and $q(x)\in \mathbb{R}[x]$. Thus, every polynomial $f$ which is non-negative on $U\times \mathbb{R}\cap\left\{(x,y)|q(x)\geq 0\right\}$ is in the preorden defined by $S\cup \left\{y-q(x)\right\}$.
\end{teorema}

\begin{proof}
The elimination of squares apply to the theorem \ref{cil} because the elements of $S$ do not depend on $y$. Then we can change $y$ with a $q(x)$ because of the compatibility with automorphisms.
\end{proof}

\begin{example}
With $q(x)=x^2$ we obtain for example: $x\geq 0, 1-x\geq 0$ and $y-x^2\geq 0$ with the picture:
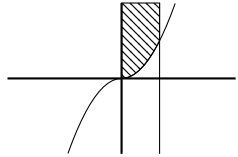
\begin{center}
\begin{tikzpicture}[scale=0.5]
\draw[thick] (-3,0) -- (3,0);
\draw[thick] (0,-2) -- (0,2);
\draw[color=black, domain=-2:2] plot (1, \x);
\draw[color=black, domain=-1.41:1.41] plot (\x,\x^2);
\draw[pattern=north west lines] plot[smooth,domain=0:1](\x, \x^2) --(1,2) -- (0,2)-- (0,0);
\end{tikzpicture}
\end{center}
Also we can apply the automorphisms in $x$ and we obtain for example with $q(y)=y^2$:\\ $x-y^2\geq 0 $ and $1-x+y^2\geq 0$:
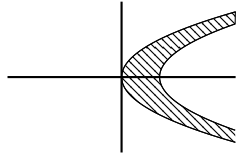
\begin{center}
\begin{tikzpicture}[scale=0.5]
\draw[thick] (-3,0) -- (3,0);
\draw[thick] (0,-2) -- (0,2);
\draw[color=black, domain=-1.73:0] plot (-\x^2, \x);
\draw[color=black, domain=0:1.73] plot (\x^2, \x);
\draw[color=black, domain=-1.41:0] plot (-\x^2+1, \x);
\draw[color=black, domain=0:1.41] plot (\x^2+1, \x);
\draw[pattern=north west lines] plot[smooth,domain=-1.73:0](-\x^2, \x) -- plot[smooth,domain=0:1.73] (\x^2, \x) -- plot[smooth,domain=1.41:0] (\x^2+1, \x) -- plot[smooth,domain=0:-1.41] (-\x^2+1, \x);
\end{tikzpicture}
\end{center}
\end{example}

In the article \cite{MW} is another substitution technique which comes from Scheiderer. Because of that substitution, we have denominators in the representation, but one can cancel them like in the part of stability.

\subsection{Non-degenerated sets}

In a new work of Vuy and Toan \cite{VT}, there is a Positivstellensatz, which applies to non-degenerated sets. The idea of the paper is searching transformations, which transform non-compact sets to compact sets, apply the theorem for the compact set and write the theorem in the previous conditions. The transformations, which have been chosen, are linear transformations between the exponents such that one needs a condition to assure, that an inverse map exists in the natural numbers. They can assure this with unimodularity at least for polynomials, which are bounded on the set, and when the set is non-degenerated. The property to be non-degenerated, one can express with a relation between the growth of the polynomials on the set and their exponents.\\
The theorem they translate, to achieve new examples in dimension 2, is a theorem of Scheiderer \cite{S5}. Here we will present a result, which already has the property to be non-degenerated:

\begin{defi}
Let $\alpha_1,\dots, \alpha_k$ be a finite set of vectors of $\mathbb{N}_0^2$ and let $r_1,\dots, r_k \in \mathbb{R}^+$. A semialgebraic set of the form
$$X^{2\alpha_1}\leq r_1, \dots, X^{2\alpha_k}\leq r_k$$
is called logarithmic polyhedron of dimension $2$.
\end{defi}

\begin{defi}
Let $C:=\left\{\lambda_1 \alpha_1+\dots +\lambda_k \alpha_k| \lambda_1,\dots, \lambda_k\geq 0\right\}$ be the cone in $\mathbb{R}_+^2$ defined by $\alpha_1,\dots, \alpha_k$. If exist $\beta_1, \beta_2\in \mathbb{N}_0^2$ with $det(\beta_1 \beta_2)=1$ and which generate the cone we call $C$ unimodular.
\end{defi}

\begin{teorema}
Let $S$ be a set of polynomials, which describe a logarithmic polyhedron of dimension $2$ like in the definition. If the cone defined by the vectors of the exponents of $S$ is unimodular and when we do not have $3$ curves $X^{2\alpha_i}= r_i$ which intersect in one point, then each polynomial $f$, which is non-negative on $K_S$, is in $T_S$.
\end{teorema}

\begin{proof}
\cite{VT}
\end{proof}

\begin{example} We will apply the elimination of squares to the example of the paper of Vuy and Toan to obtain more examples:\\
The example of the paper is $y^2\leq 1, x^2 y^2\leq 1$. The corresponding cone is defined by $(0,2)$ and $(2,2)$ and to show that it is unimodular we can choose $(1,1)$ and $(0,1)$. Thus, with the theorem, each non-negative polynomial on the set is in the corresponding preorden.
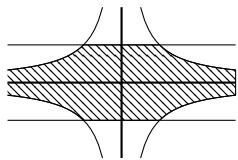
\begin{center}
\begin{tikzpicture}[scale=0.5]
\draw[thick] (-3,0) -- (3,0);
\draw[thick] (0,-2) -- (0,2);
\draw[color=black, domain=-3:3] plot (\x, 1);
\draw[color=black, domain=-3:3] plot (\x, -1);
\draw[color=black, domain=-3:-0.5] plot (\x,\x^-1);
\draw[color=black, domain=0.5:3] plot (\x,\x^-1);
\draw[color=black, domain=-3:-0.5] plot (\x,-\x^-1);
\draw[color=black, domain=0.5:3] plot (\x,-\x^-1);
\draw[pattern=north west lines] plot[smooth,domain=-3:-1](\x, -\x^-1) --(1,1) -- plot[smooth,domain=1:3] (\x, \x^-1) -- plot[smooth,domain=3:1] (\x, -\x^-1) -- (-1,-1) -- plot[smooth,domain=-1:-3] (\x, \x^-1);
\end{tikzpicture}
\end{center}

With the elimination of squares we also obtain:\\

\begin{tabular}{lll}
$y\leq 1, x^2 y\leq 1, y\geq 0$ \hspace{1cm} & $y^2\leq 1, x y^2\leq 1, x\geq 0$\hspace{1cm} & $y\leq 1, x y\leq 1, y\geq 0, x\geq 0$\\
& & \\
\begin{tikzpicture}[scale=0.5]
\draw[thick] (-3,0) -- (3,0);
\draw[thick] (0,-2) -- (0,2);
\draw[color=black, domain=-3:3] plot (\x, 1);
\draw[color=black, domain=-3:-0.707] plot (\x,-\x^-2);
\draw[color=black, domain=0.707:3] plot (\x,\x^-2);
\draw[pattern=north west lines] plot[smooth,domain=-3:-1](\x, -\x^-2) --(1,1) -- plot[smooth,domain=1:3] (\x, \x^-2) -- (3,0) -- (-3,0);
\end{tikzpicture}&
\begin{tikzpicture}[scale=0.5]
\draw[thick] (-3,0) -- (3,0);
\draw[thick] (0,-2) -- (0,2);
\draw[color=black, domain=-3:3] plot (\x, 1);
\draw[color=black, domain=-3:3] plot (\x, -1);
\draw[color=black, domain=0.5:1.73] plot (\x^2,\x^-1);
\draw[color=black, domain=0.5:1.73] plot (\x^2,-\x^-1);
\draw[pattern=north west lines] (0,1) --(1,1) -- plot[smooth,domain=1:1.73] (\x^2, \x^-1) -- plot[smooth,domain=1.73:1] (\x^2, -\x^-1) -- (0,-1);
\end{tikzpicture}&
\begin{tikzpicture}[scale=0.5]
\draw[thick] (-3,0) -- (3,0);
\draw[thick] (0,-2) -- (0,2);
\draw[color=black, domain=-3:3] plot (\x, 1);
\draw[color=black, domain=0.5:3] plot (\x,\x^-1);
\draw[color=black, domain=-3:-0.5] plot (\x,\x^-1);
\draw[pattern=north west lines] (0,1) -- plot[smooth,domain=1:3](\x, \x^-1) --(3,0) -- (0,0);
\end{tikzpicture}
\end{tabular}
\end{example}

Like in the part of stability, the results are using exponents and the description is not invariant below affine transformations.

\vspace{9cm}

\begin{remark}
How we have seen in the proof of the theorem for the representations of linear polynomials \ref{lin} and in Schmüdgen's theorem as special case of its homogeneous variant, it is not every time necessary to homogenize to even degree. We can consider the affine case not only as a special case of the projective case, it can also be a special case of the homogeneous case, when the quadratic modules contain $x_0$.\\
I could not resolve the open problem mentioned in \ref{open} yet. The tools of this paper do not seem to work, but probably one can modify the original paper in an appropriate way.
\end{remark}

\vspace{1cm}

\noindent Christoph Schulze, student at Universität Leipzig, Germany\\
E-mail address: Schulze.Christoph@t-online.de


\newpage

\begin{thebibliography}{aaaa}

\bibitem[H]{H}
  Walter Habicht:
  \emph{Über die Zerlegung strikte definiter Formen in Quadrate}.
  Commentarii Mathematici Helvetici 12, 1, 317-322 (1940)

\bibitem[KM]{KM}
  Salma Kuhlmann, Murray Marshall:
  \emph{Positivity, sums of squares and the multi-dimensional moment problem}.
  Transactions of the American Mathematics Society 354, 11, 4285-4301 (2002)
	
\bibitem[LS]{LS}
  Jesús A. de Loera, Francisco Santos:
  \emph{An effective version of Pólya's theorem on positive definite forms}.
  Journal of Pure and Applied Algebra 155, 2-3, 309-310 (1996)
	
\bibitem[M1]{M1}
  Murray Marshall:
  \emph{Positive Polynomials and Sums of Squares}.
  Instituti Editoriali e Poligrafi Internationali, Pisa (2000)

\bibitem[M2]{M2}
  Murray Marshall:
  \emph{Polynomials non-negative on a strip}.
	Proceedings of the American Mathematical Society 138, 5, 1559-1567 (2010)

\bibitem[MW]{MW}
  James H. McKay, Stuart Sui-Sheng Wang:
  \emph{An elementary proof of the automorphism theorem for the polynomial ring in two variables}.
  Journal of Pure and Applied Algebra 52, 91-102 (1988)	
	
\bibitem[N]{N}
  Tim Netzer:
  \emph{Stability of Quadratic Modules}.
	manuscripta mathematica 129 (2), 251-271 (2009)
	
\bibitem[NM]{NM}
  Tim Netzer, Pinaki Mondal:
  \emph{How Fast Do Polynomials Grow on Semialgebraic Sets?}.
  to appear in Journal of Algebra	

\bibitem[NP]{NP}
  Ha Nguyen, Victoria Powers:
  \emph{Polynomials non-negative on strips and half-strips}.
	Journal of Pure and Applied Algebra 216, 10, 2225-2232 (2012)
	
\bibitem[PR]{PR}
  Victoria Powers, Bruce Reznick:
  \emph{A new bound for Pólya's Theorem with applications to polynomials positive on polyhedra}.
  Journal of Pure and Applied Algebra 164, 1-2, 221-229 (2001)

\bibitem[S1]{S1}
  Claus Scheiderer:
  \emph{Positivity and Sums of Squares: A Guide to Recent Results}.
  Springer, Emerging Applications of Algebraic Geometry 149, 271-324 (2009)	
	
\bibitem[S2]{S2}
  Claus Scheiderer:
  \emph{Sums of squares of regular functions on real algebraic varieties}.
  Transactions of the American Mathematics Society 352, 3, 1039-1069 (2000)	
	
\bibitem[S3]{S3}
  Claus Scheiderer:
  \emph{Non-existence of degree bounds for weighted sums of squares representations}.
  Journal of Complexity 21, 6, 823-844 (2005)
	
\bibitem[S4]{S4}
  Claus Scheiderer:
  \emph{A Positivstellensatz for projective real varieties}.
  manuscripta mathematica 138, 73-88 (2012) 

\bibitem[S5]{S5}
  Claus Scheiderer:
  \emph{Sums of squares on real algebraic surfaces}.
  manuscripta mathematica 119, 4, 395-410 (2006) 
	
\bibitem[VT]{VT}
  Huy-Vui Ha, Toan Minh Ho:
  \emph{Positive polynomials on nondegenerate basic semialgebraic sets}.
  Preprint
	
\bibitem[Z]{Z}
  Guangxin Zeng:
  \emph{Homogeneous Stellensätze in semialgebraic geometry}.
  Pacific Journal of Mathematics 139, 1 (1989)
  
\end{thebibliography}
\end{document}